\documentclass[12pt,a4paper]{article}

\usepackage{amsmath}
\usepackage{amsfonts}
\usepackage{amssymb}
\usepackage{amsthm}

\usepackage{csquotes}

\theoremstyle{definition}
\newtheorem{theorem}{Theorem}[section]

\newtheorem*{condition*}{Condition}

\newtheorem*{claim*}{Claim}

\newtheorem{claimproof}{Proof of Claim \theclaim}
\newtheorem*{claimproof*}{Proof of Claim \theclaim}

\newtheorem{theoremproof}{Proof of Theorem \thetheorem}
\newtheorem*{theoremproof*}{Proof of Theorem \thetheorem}

\newtheorem*{lemma*}{Lemma}

\newtheorem{remark}[theorem]{Remark}
\newtheorem{definition}[theorem]{Definition}

\newtheorem{example}[theorem]{Example}

\renewenvironment{proof}{\textbf{Proof.}}{\qed}

\renewenvironment{claimproof*}{\textbf{Proof of Claim \theclaim}}{\qed}

\renewenvironment{theoremproof*}{\textbf{Proof of Theorem \thetheorem}}{\qed}

\usepackage[dvipsnames]{xcolor}
\usepackage{tikz}
\usepackage{wrapfig}
\usepackage{graphicx}

\usepackage{tocloft}
\usepackage{tabularx}
\usepackage{multirow}
\usepackage{caption}
\usepackage{capt-of}

\usepackage{nicefrac}

\begin{document}

\title{Nowhere differentiable functions with respect to the position} 
\author{K. Kavvadias and K. Makridis}
\maketitle

\begin{abstract}

Let $\Omega$ be a bounded domain in $\mathbb{C}$ such that $\partial \Omega$ does not contain isolated points. Let $R(\Omega)$ be the space of uniform limits on $\overline{\Omega}$ of rational functions with poles off $\overline{\Omega}$, endowed with the supremum norm. We prove that either generically all functions $f$ in $R(\Omega)$ satisfy 
$$ \limsup_{\substack{z \to z_0 \\ z \in \partial \Omega}} \Big| \frac{f(z) - f(z_0)}{z - z_0} \Big| = + \infty $$

for every $z_0 \in \partial \Omega$ or no such function in $R(\Omega)$ meets this requirement. In the first case, the generic function $f \in R(\Omega)$ is nowhere differentiable on $\partial \Omega$ with respect to the position. We give specific examples where each case of the previous dichotomy holds. We also extend the previous result to unbounded domains.

\end{abstract}

AMS clasification numbers: 26A27, 30H50.

\medskip


\section{Introduction}

It is well known that the Weierstrass function $u: \mathbb{R} \to \mathbb{R}$ defined by
$$ u(x) = \sum_{n = 0}^{+ \infty} a^n cos(b^nx) $$

for $a \in (0, 1), b$ an odd integer and $ab > 1 + \frac{3 \pi}{2}$ is a $2 \pi$ - periodic function which is continuous and nowhere differentiable. Furthermore, this phenomenon is generic in the spaces $C([0, 1])$, $C(\mathbb{R})$ and $C(\mathbb{R} / 2 \pi)$. In \cite{ESKENAZIS.MAKRIDIS} and \cite{ESKENAZIS} the Weierstrass function was complexified, giving a $2 \pi$ - periodic function of analytic type and the phenomenon was proven to be generic in the disc algebra $A(D)$, where $D = \{ z \in \mathbb{C}: |z| < 1 \}$ is the open unit disc (centered at $0$). We remind that a function $f: \overline{D} \to \mathbb{C}$ belongs to $A(D)$ if and only if $f$ is continuous on $\overline{D}$ and holomorphic in $D$. More precisely, there exist functions $f \in A(D)$ such that the (real) functions $Ref(e^{i \theta})$ and $Imf(e^{i \theta})$ are nowhere differentiable with respect to the real parameter $\theta$. Furthermore, the set of these functions $f \in A(D)$ is residual in $A(D)$, where the space $A(D)$ is endowed with the supremum norm on $\overline{D}$; it contains the set of functions $f \in A(D)$such that 
$$ \limsup_{t \to t_0} \Big| \frac{Ref(e^{i t}) - Ref(e^{i t_0})}{t - t_0} \Big| = \limsup_{t \to t_0} \Big| \frac{Imf(e^{i t}) - Imf(e^{i t_0})}{t - t_0} \Big| = + \infty \;\;\;\;\;\; (*) $$

for all $t_0 \in \mathbb{R}$. The last set is $G_{\delta}$ - dense in $A(D)$. 

In \cite{MASTRANTONIS.PANAGIOTIS} a generalization of the previous result is obtained, where the open unit disc $D$ is replaced by a domain $\Omega$, where $\Omega$ is bounded by a finite set of disjoint Jordan curves. Now, the parametrization of each bounded Jordan curve is induced by a conformal Riemann mapping. It is essential that the denominator in relation $(*)$ is the quantity $t - t_0$, where $t$ is the (real) parameter of the curve $\gamma: t \mapsto \gamma(t)$; thus, the functions obtained in this way are nowhere differentiable with respect to the real parameter $t$.

In the present paper we wish to obtain analogous results, where the functions will be nowhere differentiable with respect to the position $z \in \mathbb{C}$. It suffices to prove that it holds
$$ \limsup_{\substack{z \to z_0 \\ z \in \partial \Omega}} \Big| \frac{f(z) - f(z_0)}{z - z_0} \Big| = + \infty \;\;\;\;\;\; (**) $$

for all (or some) $z_0 \in \partial \Omega$. Certainly, if the boundary of $\Omega$ is a curve with a smooth parametrization with non - vanishing derivative, then relation $(*)$ implies relation $(**)$. This yields that the phenomenon of functions $f \in A(\Omega)$ satisfying relation $(**)$ for every $z_0 \in J$, where $J \subseteq \partial \Omega$ is a compact set without isolated points, is residual in $A(\Omega)$; here $\Omega$ is a disc, a half plane or the complement of a disc or even other domains related to the previous ones. Taking advantage of Baire's Category Theorem we prove, for instance, that this phenomenon is generic when $\Omega$ is a bounded angular sector. Several such examples are given in Section 4 below.

The main result is proven in Section 2 (Theorem \ref{main theorem of 2}). It states that if $\Omega$ is a bounded domain in $\mathbb{C}$, $J \subseteq \partial \Omega$ is a compact set without isolated points and $S(\Omega, J)$ denotes the set of functions $f \in R(\Omega)$, satisfying  
$$ \limsup_{\substack{z \to z_0 \\ z \in J}} \Big| \frac{f(z) - f(z_0)}{z - z_0} \Big| = + \infty \;\;\;\;\;\; (***) $$

for all $z_0 \in J$, then the class $S(\Omega, J)$ is either void or $G_{\delta}$ - dense in $R(\Omega)$. Here, $R(\Omega)$ denotes the space of uniform limits on $\overline{\Omega}$ of rational functions with poles off $\overline{\Omega}$, endowed with the supremum norm on $\overline{\Omega}$. Quite often, it holds $R(\Omega) = A(\Omega)$; see the relevant comments in Section 3.

In order to generalize the previous result to unbounded domains $\Omega$, V. Nestoridis suggested to replace $R(\Omega)$ with the space $\widetilde{R}(\Omega)$. The space $\widetilde{R}(\Omega)$ consists of all functions $f: \overline{\Omega} \to \mathbb{C}$, where the closure $\overline{\Omega}$ is taken in $\mathbb{C}$ and apparently does not contain $\infty$, which are uniform limits on each compact subset of $\overline{\Omega}$ of rational functions with poles off $\overline{\Omega}$. This is a Fr\'{e}chet space, endowed with the seminorms 
$$ \sup_{\substack{z \in \overline{\Omega} \\ |z| \leq m}} |f(z)|, f \in \widetilde{R}(\Omega) \; \text{for} \; m = 1, 2, \cdots. $$

Then again the class $S(\Omega, J)$ is either void or $G_{\delta}$ -dense in $\widetilde{R}(\Omega)$. This is the content of Section 3. Also, quite often it holds $\widetilde{R}(\Omega) = A(\Omega)$.

In Section 4 several examples are given. In one such example it holds $S(\Omega, J) = \emptyset$, but we do not have an analogous example if in addition it holds ${\overline{\Omega}}^{\circ} = \Omega$. In every other example in Section 4 it holds $S(\Omega, J) \neq \emptyset$ and the generic function $f$ in $A(\Omega)$ is nowhere differentiable with respect  to the position on $\partial \Omega$.

\section{A dichotomy result for bounded domains}

Let $\Omega$ be a bounded domain in $\mathbb{C}$. We denote by $R(\Omega)$ the set of uniform limits on $\overline{\Omega}$ of rational functions with poles off $\overline{\Omega}$. The space $R(\Omega)$ is a Banach space endowed with the supremum norm on $\overline{\Omega}$. 

More generally, if $K \subseteq \mathbb{C}$ is a compact set, then $R(K)$ is the set of uniform limits on $K$ of rational functions with poles off $K$, endowed with the supremum norm on $K$. The space $R(K)$ is also a Banach space. Obviously, it holds $R(\Omega) = R(\overline{\Omega})$ for every bounded domain $\Omega \subseteq \mathbb{C}$.

Let $\Omega \subseteq \mathbb{C}$ be a bounded domain and $J \subseteq \partial \Omega$ be a compact set without isolated points. We denote with $S(\Omega, J)$ the following class of functions 
$$ S(\Omega, J) = \{f \in R(\Omega): \limsup_{\substack{z \to z_0 \\ z \in J \setminus \{ z_0 \}}} \Big| \frac{f(z) - f(z_0)}{z - z_0} \Big| = + \infty \; \text{for every} \; z_0 \in J \}. $$ 

\begin{theorem} \label{main theorem of 2} Under the above assumptions and notations, the class $S(\Omega, J)$ is either void or $G_{\delta}$ - dense in $R(\Omega)$.

\end{theorem}

\begin{proof} We suppose that it holds $S(\Omega, J) \neq \emptyset$ and let $f \in S(\Omega, J)$. We denote with $E_n$ the following set 
$$ E_n = \{g \in R(\Omega): \; \text{for every} \; z_0 \in J \; \text{there exists a} \; z \in ( J \setminus \{ z_0 \}) \cap D(z_0, \frac{1}{n}) $$
$$ \text{such that} \; \Big| \frac{f(z) - f(z_0)}{z - z_0} \Big| > n \} $$ 

where $D(z_0, \frac{1}{n})$ denotes the open  disc centered at $z_0$ and with radius $\frac{1}{n}$. 

Obviously, it holds
$$ S(\Omega, J) = \bigcap_{n = 1}^{+ \infty} E_n. $$

Firstly we prove that each $E_n$ is an open subset of $(R(\Omega), || \cdot ||_{\infty})$, or equivalently, that each $R(\Omega) \setminus E_n$ is a closed set. Indeed, let $\{ g_m \}_{m \geq 1} \subseteq R(\Omega) \setminus E_n$ and $g \in R(\Omega)$ such that $g_m \to g$. Then, for every $m \geq 1$, there exists a $z_m \in J$ satisfying  
$$ \Big| \frac{g_m(z) - g_m(z_m)}{z - z_m} \Big| \leq n $$

for every $z \in ( J \setminus \{ z_m \}) \cap D(z_m, \frac{1}{n})$. Since $J$ is a compact set, there exists a subsequence of $\{ z_m \}_{m \geq 1}$ which converges to a single point $z_0 \in J$. Without loss of generality, we may assume that $\{ z_m \}_{m \geq 1}$ converges to $z_0$. Let $z \in ( J \setminus \{ z_0 \}) \cap D(z_0, \frac{1}{n})$ be a fixed point. Then, there exists an index $m_0 \in \mathbb{N}$ satisfying $z \in ( J \setminus \{ z_m \}) \cap D(z_m, \frac{1}{n})$ for every $m \geq m_0$. Consequently, for every $m \geq m_0$ it holds 
$$ |g(z) - g(z_0)| \leq |g(z) - g_m(z)| + |g_m(z) -  g_m(z_m)| + |g_m(z_m) - g(z_0)| \leq $$
$$ \leq |g(z) - g_m(z)| + |g_m(z_m) -  g(z_0)| + n|z_m - z|. $$

Therefore, by taking limits in the previous relation as $m \to + \infty$ we obtain that it holds $|g(z) - g(z_0)| \leq n |z - z_0|$ for every $z \in ( J \setminus \{ z_0 \}) \cap D(z_0, \frac{1}{n})$, because $g_m \to g$ uniformly and $z_m \to z_0$.

Thus, $g \in R(\Omega) \setminus E_n$ and as a result, $E_n$ is a closed set. Therefore, its complement is an open set for every $n \geq 1$.

It remains to prove the density of $S(\Omega, J)$. Let $g \in R(\Omega)$. Then, there exists a rational function $q \equiv q_{\varepsilon}$ with poles off $\overline{\Omega}$ such that $|| (g - f) - q ||_{\infty} < \varepsilon$ for a fixed $\varepsilon >0$. Since $q'$ is continuous on $\overline{\Omega}$, there exists $M < + \infty$ satisfying $||q'||_{\infty} \leq M$. Considering a fixed point $z_0 \in J$ we can find a sequence $\{ z_m \}_{m \geq 1}$ in $J \setminus \{ z_0\}$ such that $z_m \to z_0$ satisfying
$$ \lim_{m \to + \infty} \Big| \frac{f(z_m) - f(z_0)}{z_m - z_0} \Big| = + \infty. $$

By the triangle inequality, it holds 
$$ \Big| \frac{(f + q)(z_m) - (f + q)(z_0)}{z_m - z_0} \Big| \geq \Big| \frac{f(z_m) - f(z_0)}{z_m - z_0} \Big| - \Big| \frac{q(z_m) - q(z_0)}{z_m - z_0} \Big|. $$

At the same time, it also holds
$$ \lim_{m \to + \infty} \Big| \frac{q(z_m) - q(z_0)}{z_m - z_0} \Big| = |q'(z_0)| \leq M. $$

Therefore, by combining the previous two relations, we obtain
$$ \lim_{m \to + \infty} \Big| \frac{(f + q)(z_m) - (f + q)(z_0)}{z_m - z_0} \Big| = + \infty $$ 

and thus  
$$ \limsup_{\substack{z \to z_0 \\ z \in J \setminus \{ z_0 \}}} \Big| \frac{(f + q)(z) - (f + q)(z_0)}{z - z_0} \Big| = + \infty $$ 

for every $z_0 \in J$. Consequently, we deduce that $(f +  q) \in S(\Omega , J)$. Since $||g - (f + q)||_{\infty} < \varepsilon$ and $\varepsilon > 0$ is arbitrary, it follows that $g \in \overline{S(\Omega , J)}$. Thus, we have proved the density of $S(\Omega , J)$. Baire's Theorem completes the proof.

\end{proof}

Let $K \subseteq \mathbb{C}$ be a compact set. Then, we denote with $A(K)$ the set of all functions $f:K \to \mathbb{C}$ which are continuous on $K$ and holomorphic in $K^{\circ}$ (if $K^{\circ} = \emptyset$, then $A(K) = C(K)$). We endow the space $A(K)$ with the supremum norm on $K$ and thus, $A(K)$ becomes a Banach space. If $\Omega$ is a bounded domain in $\mathbb{C}$, then $R(\Omega)$ is a closed  subspace of $A(\overline{\Omega})$.

\begin{definition} [Sets of approximation] We say that a compact set $K \subseteq \mathbb{C}$ is a set of approximation if it holds $A(K) = R(K)$.

\end{definition}

A characterization of whether a compact set $K$ is a set of approximation is given in \cite{VITUSHKIN}, using the notion of continuous analytic capacity. Furthermore, sufficient conditions ensuring that a compact set $K$ is a set of approximation are the following:

\begin{itemize}

\item[$(i)$] $(\mathbb{C} \cup \{ \infty \}) \setminus K$ has finitely many connected components (\cite{RUDIN}, exercise 1, chapter 20, page  394)

\item[$(ii)$] The diameters of the connected components of the complement of $K$ (even though there may be infinitely many such components) are uniformly bounded away from zero (\cite{GAMELIN}).

\item[$(iii)$] $K$ has area measure zero (\cite{GAMELIN}).

\end{itemize}

\begin{remark} If $\Omega \subseteq \mathbb{C}$ is a bounded domain in $\mathbb{C}$ such that $\overline{\Omega}$ is a compact set of approximation, then Theorem \ref{main theorem of 2} ensures that the class $S(\Omega, J)$ is either void or $G_{\delta}$ - dense in $A(\overline{\Omega}) = R(\Omega)$. We define $A(\Omega)$ to be the class of functions $f: \overline{\Omega} \to \mathbb{C}$ continuous on $\overline{\Omega}$ and holomorphic in $\Omega$, endowed with the supremum norm on $\overline{\Omega}$. This is also a Banach space and $A(\overline{\Omega}) \subseteq A(\Omega)$ but we do not always have the equality $A(\overline{\Omega}) = A(\Omega)$. If $(\overline{\Omega})^{\circ} = \Omega$; that is, if $\Omega$ is a Carath\'{e}odory domain, then we have the equality $A(\overline{\Omega}) = A(\Omega)$ and in this case Theorem \ref{main theorem of 2} ensures that $S(\Omega, J)$ is either void or $G_{\delta}$ - dense in $A(\Omega)$.

\end{remark}

\section{Unbounded domains}

Let $E \subseteq \mathbb{C}$ be an unbounded open set. We denote with $\widetilde{R}(E)$ the set of all functions which are uniform limits on each compact subset of $\overline{E}$ of rational functions with poles off $\overline{E}$. The natural topology of $\widetilde{R}(E)$ is the topology of uniform convergence on each compact subset of $\overline{E}$. Equivalently, it is defined by the sequence of seminorms 
$$ \sup_{\substack{z \in \overline{E} \\ |z| \leq n}} |f(z)|, f \in \widetilde{R}(E) \; \text{for} \; n = 1, 2, \cdots. $$

Moreover the space $\widetilde{R}(E)$ endowed with these seminorms is a Fr\'{e}chet space.

\begin{theorem} \label{main theorem of 3} Let $\Omega \subseteq \mathbb{C}$ be an unbounded domain and $J \subseteq \partial \Omega$ be a compact set without isolated points. Then, the class 
 
$$ S(\Omega, J) = \{f \in \widetilde{R}(\Omega): \limsup_{\substack{z \to z_0 \\ z \in J \setminus \{ z_0 \}}} \Big| \frac{f(z) - f(z_0)}{z - z_0} \Big| = + \infty \; \text{for every} \; z_0 \in J \} $$ 

is either void or $G_{\delta}$ - dense in $\widetilde{R}(\Omega)$. 

\end{theorem}

\begin{proof} We suppose that $S(\Omega, J) \neq \emptyset$ and let $f \in S(\Omega, J)$. We denote with $E_n$ the following set 
$$ E_n = \{g \in \tilde{R}(\Omega): \; \text{for every} \; z_0 \in J \; \text{there exists a} \; z \in ( J \setminus \{ z_0 \}) \cap D(z_0, \frac{1}{n}) $$
$$ \text{such that} \; \Big| \frac{g(z) - g(z_0)}{z - z_0} \Big| > n \} $$ 

Obviously, it holds
$$ S(\Omega, J) = \bigcap_{n = 1}^{+ \infty} E_n. $$

Firstly we show that each $E_n$ is an open subset of $\widetilde{R}(\Omega)$, or equivalently, that each $\widetilde{R}(\Omega) \setminus E_n$ is closed in $\widetilde{R}(\Omega)$. Indeed, let $\{ g_m \}_{m \geq 1}$ be a sequence of functions in $\widetilde{R}(\Omega) \setminus E_n$ which converges uniformly on the compact subsets of $\overline{\Omega}$ to a function $g \in \widetilde{R}(\Omega)$. Then, for every $m \geq 1$, there exists a $z_m \in J$ satisfying  
$$ \Big| \frac{g_m(z) - g_m(z_m)}{z - z_m} \Big| \leq n $$

for every $z \in ( J \setminus \{ z_m \}) \cap D(z_m, \frac{1}{n})$. Since $J$ is a compact set, there exists a subsequence of $\{ z_m \}_{m \geq 1}$ which converges to a single point $z_0 \in J$. Without loss of generality, we may assume that $\{ z_m \}_{m \geq 1}$ converges to $z_0$. Let $z \in ( J \setminus \{ z_0 \}) \cap D(z_0, \frac{1}{n})$ be a fixed point. Then, there exists an index $m_0 \in \mathbb{N}$ satisfying $z \in ( J \setminus \{ z_m \}) \cap D(z_m, \frac{1}{n})$ for every $m \geq m_0$. Consequently, for every $m \geq m_0$ it holds 
$$ |g(z) - g(z_0)| \leq |g(z) - g_m(z)| + |g_m(z) -  g_m(z_m)| + |g_m(z_m) - g(z_0)| \leq $$
$$ \leq |g(z) - g_m(z)| + |g_m(z_m) -  g(z_0)| + n|z_m - z|. $$

Furthermore, the sequence $\{ g_m \}_{m \geq 1}$ converges uniformly to $g \in \widetilde{R}(\Omega)$ on $J$, since $J \subseteq \overline{\Omega}$ is a compact set. Therefore, by taking limits in the previous relation as $m \to + \infty$ we obtain that $|g(z) - g(z_0)| \leq n |z - z_0|$ for every $z \in ( J \setminus \{ z_0 \}) \cap D(z_0, \frac{1}{n})$.

Thus $g \in \widetilde{R}(\Omega) \setminus E_n$ and as a result, $E_n$ is a closed set. Therefore, its complement is an open set for every $n \geq 1$.

It remains to prove the density of $S(\Omega, J)$. Let $g \in \widetilde{R}(\Omega)$. Then, there exists a sequence of rational functions $\{ q_m \}_{m \geq 1}$ with poles off $\overline{\Omega}$ which converges uniformly to the function $g - f$ on each compact subset of $\overline{\Omega}$. Obviously, the sequence $\{ f + q_m \}_{m \geq 1}$ converges uniformly to $g$ on the compacts subsets of $\overline{\Omega}$. Now, let $q$ be a rational function with poles off $\overline{\Omega}$. Since $q'$ is continuous on the compact set $J \subseteq \partial \Omega$, there exists a $M < + \infty$ such that $|q'(z)| \leq M$ for every $z \in J$. Let $z_0 \in J$; since it holds
$$ \limsup_{\substack{z \to z_0 \\ z \in J \setminus \{ z_0 \}}} \Big| \frac{f(z) - f(z_0)}{z - z_0} \Big| = + \infty $$

we can find a sequence $\{ z_m \}_{m \geq 1}$ such that $z_m \in (J \setminus \{ z_0 \}) \cap B(z_0, \frac{1}{n})$ for every $m \geq 1$, also $z_m \to z_0$ and 
$$ \lim_{m \to + \infty} \Big| \frac{f(z_m) - f(z_0)}{z_m - z_0} \Big| = + \infty. $$

At the same time, it also holds 
$$ \Big| \frac{(f + q)(z_m) - (f + q)(z_0)}{z_m - z_0} \Big| \geq \Big| \frac{f(z_m) - f(z_0)}{z_m - z_0} \Big| - \Big| \frac{q(z_m) - q(z_0)}{z_m - z_0} \Big| $$

and obviously it holds
$$ \lim_{m \to + \infty} \Big| \frac{q(z_m) - q(z_0)}{z_m - z_0} \Big| = |q'(z_0)| \leq M. $$

Therefore, by combining the previous two relations, we obtain
$$ \lim_{m \to + \infty} \Big| \frac{(f + q)(z_m) - (f + q)(z_0)}{z_m - z_0} \Big| = + \infty $$

and thus 
$$ \limsup_{\substack{z \to z_0 \\ z \in J \setminus \{ z_0 \}}} \Big| \frac{(f + q)(z) - (f + q)(z_0)}{z - z_0} \Big| = + \infty $$

for every $z_0 \in J$.Therefore, we can deduce that $(f + q_n) \in S(\Omega, J)$ for every $n \geq 1$ and thus, since $f + q_n \to g$ in the topology of $\widetilde{R}(E)$, it follows that $g \in \overline{S(\Omega, J)}$. Thus, we have proved the density of $S(\Omega, J)$. Baire's Theorem completes the proof.

\end{proof}

If $E \subseteq \mathbb{C}$ is a closed set, then by $A(E)$ we denote the set of all functions $f:E \to \mathbb{C}$ which are continuous on $E$ and holomorphic in $E^{\circ}$ (if $E^{\circ} = \emptyset$, then $A(E) = C(E)$). In any case, $A(E)$ is a Fr\'{e}chet space when endowed with the seminorms 
$$ \sup_{\substack{z \in E \\ |z| \leq n}} |f(z)|, f \in A(E) \; \text{for} \; n = 1, 2, \cdots. $$

In addition, if $\Omega \subseteq \mathbb{C}$ is an unbounded domain and $\overline{\Omega} \cap \overline{D(0, n)}$ is a compact set of approximation for every $n \geq 1$, then $\widetilde{R}(\Omega) = A(\overline{\Omega})$. Indeed, obviously it holds that $\widetilde{R}(\Omega) \subseteq A(\overline{\Omega})$. Let $g \in A(\overline{\Omega})$ and $n \geq 1$. Since $\overline{\Omega} \cap \overline{D(0, n)}$ is a compact set of approximation, there exists a rational function $\phi_n$ with poles off $X_n = \overline{\Omega} \cap \overline{D(0, n)}$, such that 
$$ \sup_{z \in X_n} |g(z) - \phi_n(z)| < \frac{1}{2n}. $$

The function $\phi_n$ is holomorphic in an open set containing the compact set $X_n$. Now, let $B \subseteq (\mathbb{C} \cup \{ \infty \}) \setminus X_n$ be a set that intersects each component of $(\mathbb{C} \cup \{ \infty \}) \setminus X_n$. We also assume that the set $B$ satisfies $B \cap L = \{ \infty \}$, where $L$ is the component of $(\mathbb{C} \cup \{ \infty \}) \setminus X_n$ containing $\infty$. From Runge's theorem, which can be applied because $\phi_n$ is holomorphic in an open set containing $X_n$, we can find a rational function $g_n$ with poles in $B$ such that 
$$ \sup_{X_n} |\phi_n(z) - g_n(z)| < \frac{1}{2n}. $$

Obviously, $||g - g_n ||_{\infty} < \frac{1}{n}$ on $X_n = \overline{\Omega} \cap \overline{B(0, n)}$. Now, let $r$ be a pole of $g_n$. There are two possible cases.

\begin{itemize}

\item[$\bullet$] If $r \in \overline{D(0, n)}$, then $r \not\in \overline{\Omega}$ since $r \not\in X_n = \overline{\Omega} \cap \overline{D(0, n)}$.                 

\item[$\bullet$] If $r \not\in \overline{D(0, n)}$, then $r \in \Big[ (\mathbb{C} \cup \{ \infty \}) \setminus X_n \Big] \cap L$ and thus, $r = \infty \not\in \overline{\Omega}$. 

\end{itemize}

In every case, we have that it holds $r \not\in \overline{\Omega}$. Consequently, for every $n \geq 1$, there exists a rational function $g_n$ with poles off $\overline{\Omega}$ such that $|g(z) - g_n(z)| < \frac{1}{n}$ for every $z \in \overline{\Omega} \cap \overline{D(0, n)}$. Moreover, the sequence $\{ g_n \}_{n \geq 1}$ converges uniformly to $g$ on the compact subsets of $\overline{\Omega}$ and thus, $g \in \widetilde{R}(\Omega)$. Therefore, we obtain that it holds $\widetilde{R}(\Omega) = A(\overline{\Omega})$.

However, we do not know if the converse holds; that is, if for an unbounded domain $\Omega \subseteq \mathbb{C}$ it holds $\widetilde{R}(\Omega) = A(\overline{\Omega})$, is it necessarily true that $\overline{\Omega} \cap \overline{D(0, \lambda_n)}$ is a compact set of approximation for all elements $\lambda_n$ of a sequence $\{ \lambda_n \}_{n \geq 1} \subseteq \mathbb{R}$ converging to $+ \infty$?

\section{Examples}

In this section we give specific examples. In each such example we clarify if the class $S(\Omega, J)$ is void or $G_{\delta}$ - dense.

\begin{example} \label{first example} Let $\Omega = D$ be the open unit disc and $\mathbb{T} = \partial D$. We consider the function $g: D \to \mathbb{C}$ defined as follows
$$ g(z) = \sum_{n = 0}^{+ \infty} a^n z^{b^n} \; \text{for every} \; z \in D. $$ 

where $a \in (0, 1), b$ an odd integer and $ab > 1 + \frac{3 \pi}{2}$. Then, it is easy to see that it holds $Re(g) = u$, where $u$ is the Weierstrass function (see also \cite{ESKENAZIS}, \cite{ESKENAZIS.MAKRIDIS}, \cite{MASTRANTONIS.PANAGIOTIS}). It is known that it holds $g \in A(D)$. In addition,
$$ \limsup_{\substack{\theta \to \theta_0 \\ \theta \in (\theta_0 - 1, \theta_0 + 1) \setminus \{ \theta_0 \}}} \Big| \frac{g(e^{i \theta}) - g(e^{i \theta_0})}{\theta - \theta_0} \Big| = + \infty $$

for every $\theta_0 \in [0, 2 \pi]$. We consider the following class of functions 
$$ S(D, \mathbb{T}) = \{ f \in A(D): \limsup_{\substack{z \to z_0 \\ z \in \mathbb{T} \setminus \{ z_0 \}}} \Big| \frac{f(z) - f(z_0)}{z - z_0} \Big| = + \infty \; \text{for every} \; z_0 \in \mathbb{T} \}. $$

Then, since $A(D) = R(\overline{D})$ it either holds $S(D, \mathbb{T}) = \emptyset$ or $S(D, \mathbb{T})$ is $G_{\delta}$ - dense in $(A(\Omega), || \cdot ||_{\infty})$, according to Theorem \ref{main theorem of 2}. Let $z_0 = e^{i \theta_0} \in \mathbb{T}$. We consider a sequence $\{ \theta_n \}_{n \geq 1}$ in $\mathbb{R} \setminus \{ \theta_0 \}$ with $\theta_n \to \theta_0$ satisfying
$$ \lim_{n \to + \infty} \Big| \frac{g(e^{i \theta_n}) - g(e^{i \theta_0})}{\theta_n - \theta_0} \Big| = + \infty. $$

In addition, for every $n \geq 1$ we set $z_n = e^{i \theta_n} \in \mathbb{T}$. Then, $z_n \to z_0$ and there exists an index $n_0 \in \mathbb{N}$ such that $z_n \neq z_0$ for every $n \geq n_0$. For every $n \geq n_0$ we have
$$ \frac{g(z_n) - g(z_0)}{z_n - z_0} = \frac{g(e^{i \theta_n}) - g(e^{i \theta_0})}{\theta_n - \theta_0} \cdot \frac{\theta_n - \theta_0}{e^{i \theta_n} - e^{i \theta_0}}. $$

Since
$$ \lim_{n \to + \infty} \frac{\theta_n - \theta_0}{e^{i \theta_n} - e^{i \theta_0}} = \frac{1}{i e^{i \theta_0}} \neq 0 $$

we obtain that it holds
$$ \lim_{n \to + \infty} \Big| \frac{g(z_n) - g(z_0)}{z_n - z_0} \Big| = + \infty $$

and therefore
$$ \limsup_{\substack{z \to z_0 \\ z \in \mathbb{T} \setminus \{ z_0 \}}} \Big| \frac{g(z) - g(z_0)}{z - z_0} \Big| = + \infty $$

for every $z_0 \in \mathbb{T}$. The previous relation implies that $g \in S(D, \mathbb{T})$ and thus $S(D, \mathbb{T}) \neq \emptyset$; according to Theorem \ref{main theorem of 2}, the class $S(D, \mathbb{T})$ is $G_{\delta}$ - dense in $A(D)$. 

\end{example}

\begin{example} \label{second example}

Let $\Omega = \{ z \in \mathbb{C}: Re(z) > 0 \}$. For every $n \geq 1$ we consider the following classes of functions 
$$ S(\Omega, J_n) = \{ f \in A(\Omega): \limsup_{\substack{z \to z_0 \\ z \in J_n \setminus \{ z_0 \}}} \Big| \frac{f(z) - f(z_0)}{z - z_0} \Big| = + \infty \; \text{for every} \; z_0 \in J_n \} $$

where $J_n = [- in, + in]$. In addition, let
$$ S(\Omega, J) = \{ f \in A(\Omega): \limsup_{\substack{z \to z_0 \\ z \in J \setminus \{ z_0 \}}} \Big| \frac{f(z) - f(z_0)}{z - z_0} \Big| = + \infty \; \text{for every} \; z_0 \in J \} $$

where 
$$ J = \bigcup_{n = 1}^{+ \infty} J_n = i \mathbb{R}. $$ 

Then, the class $S(\Omega, J)$ is $G_{\delta}$ - dense in $A(\Omega)$, where the space $A(\Omega)$ is endowed with the topology of uniform convergence on the compact subsets of $\overline{\Omega}$. 

\end{example}

\begin{proof} We prove that each $S(\Omega, J_n)$ is $G_{\delta}$ - dense in $A(\Omega) = \widetilde{R}(\Omega)$. Since $(\mathbb{C} \cup \{ \infty \}) \setminus (\overline{\Omega} \cap \overline{D(0, n)})$ is a connected set, it follows that $A((\overline{\Omega})^{\circ} \cap D(0, n)) = R(\overline{\Omega} \cap \overline{D(0, n)})$ for every $n \geq 1$ and thus, according to Section 3, it is enough to prove that each $S(\Omega, J_n)$ is non - void.

We consider the entire function $\phi: \mathbb{C} \to \mathbb{C}$ satisfying $\phi(w) = e^{- w}$ for every $w \in \mathbb{C}$. Obviously, $\phi'(w) = - e^{- w} \neq 0$ for every $w \in \mathbb{C}$. We also consider the function $f: \Omega \to \mathbb{C}$ with $f = g \circ \phi$, where $g$ is the function defined in Example \ref{first example}. Obviously, $\phi(\Omega) \subseteq D(0, 1)$ and $\phi(J_n) \subseteq \mathbb{T}$. The reader can easily verify that it holds $f \in S(\Omega, J_n)$ for every $n \geq 1$. Thus, the class $S(\Omega, J_n)$ is $G_{\delta}$ - dense in $A(\Omega)$ and since it holds 
$$ S = \bigcap_{n = 1}^{+ \infty} S(\Omega, J_n) $$

Baire's Theorem implies that the class $S(\Omega, J)$ is $G_{\delta}$ - dense in $A(\Omega)$.

\end{proof}

\begin{remark} Example \ref{second example} shows that in the definition of the class $S(\Omega, J)$ in Section 3 the set $J$ may not be a compact one.

\end{remark}

\begin{example} \label{third example}

We consider the following sets 
$$ A = \{ r e^{i \frac{3 \pi}{4}}: 0 \leq r \leq 1 \} $$
$$ B = \{ e^{i \theta}: \frac{\pi}{4} \leq \theta \leq \frac{3 \pi}{4} \} $$ 

and 
$$ C = \{ r e^{i \frac{\pi}{4}}: 0 \leq r \leq 1\}. $$ 

Let $\Omega$ be the Jordan domain bounded by $A \cup B \cup C$. It clearly holds $\partial \Omega = A \cup B \cup C$. Then, the class of functions
$$ S(\Omega, \partial \Omega) = \{ f \in A(\Omega) : \limsup_{\substack{z \to z_0 \\ z \in \partial \Omega \setminus \{ z_0 \}}} \Big| \frac{f(z) - f(z_0)}{z - z_0} \Big| = + \infty \; \text{for every} \; z_0 \in \partial \Omega \}. $$

is $G_{\delta}$ - dense in $(A(\Omega), || \cdot ||_{\infty})$.

\end{example}

\begin{proof} We consider the following classes of functions
$$ S(\Omega, A) = \{ f \in A(\Omega) : \limsup_{\substack{z \to z_0 \\ z \in A \setminus \{ z_0 \}}} \Big| \frac{f(z) - f(z_0)}{z - z_0} \Big| = + \infty \; \text{for every} \; z_0 \in A \} $$
$$ S(\Omega, B) = \{ f \in A(\Omega) : \limsup_{\substack{z \to z_0 \\ z \in B \setminus \{ z_0 \}}} \Big| \frac{f(z) - f(z_0)}{z - z_0} \Big| = + \infty \; \text{for every} \; z_0 \in B \} $$
$$ S(\Omega, C) = \{ f \in A(\Omega) : \limsup_{\substack{z \to z_0 \\ z \in C \setminus \{ z_0 \}}} \Big| \frac{f(z) - f(z_0)}{z - z_0} \Big| = + \infty \; \text{for every} \; z_0 \in C \}. $$

Obviously, $S(\Omega, A) \cap S(\Omega, B) \cap S(\Omega, C) \subseteq S(\Omega, \partial \Omega)$. From Example \ref{second example}, there exists a function $h \in A(R)$ (where $R$ is the open right half plane) satisfying
$$ \limsup_{\substack{z \to z_0 \\ z \in i \mathbb{R} \setminus \{ z_0 \}}} \Big| \frac{h(z) - h(z_0)}{z - z_0} \Big| = + \infty $$

for every $z_0 \in i \mathbb{R}$. In the same way, one can prove that there exists a function $\phi \in A(L)$ (where $L$ is the open left half plane) satisfying
$$ \limsup_{\substack{z \to z_0 \\ z \in i \mathbb{R} \setminus \{ z_0 \}}} \Big| \frac{\phi(z) - \phi(z_0)}{z - z_0} \Big| = + \infty $$

for every $z_0 \in i \mathbb{R}$.

We consider the functions $\omega_1, \omega_2 \in A(\Omega)$ with 
$$ \omega_1(z) = h(e^{-i \frac{\pi}{4} z}) $$ 
and 
$$ \omega_2(z) = \phi(e^{i \frac{\pi}{4} z}) $$ 

for every $z \in \Omega$. We have proved that $\omega_1 \in S(\Omega, A)$ and $\omega_2 \in S(\Omega, C)$; thus, according to Section 2, the classes $S(\Omega, A)$ and $S(\Omega, C)$ are $G_{\delta}$ - dense in $A(\Omega)$. In addition, if $g$ is the function defined in Example \ref{first example}, then $(g \restriction_{\Omega}) \in S(\Omega, B)$ and therefore, the class $S(\Omega, B)$ is also $G_{\delta}$ - dense in $A(\Omega)$. According to Baire's Theorem, it follows that the class $S(\Omega, A) \cap S(\Omega, B) \cap S(\Omega, C)$ is $G_{\delta}$ - dense in $A(\Omega)$. Since $S(\Omega, A) \cap S(\Omega, B) \cap S(\Omega, C) \subseteq S(\Omega, J)$, we obtain that it holds $S(\Omega, J) \neq \emptyset$ and thus, the class $S(\Omega, J)$ is also $G_{\delta}$ - dense in $A(\Omega)$.

\end{proof}

\begin{example} \label{fourth example}

Let $\Omega$ be the interior of a convex polygonal domain, with 
$$ \partial \Omega = [a_0, a_1] \cup [a_1, a_2] \cup \cdots \cup [a_n, a_0]. $$ 

Then, the class of functions 
$$ S(\Omega, \partial \Omega) = \{ f \in A(\Omega) : \limsup_{\substack{z \to z_0 \\ z \in \partial \Omega \setminus \{ z_0 \}}} \Big| \frac{f(z) - f(z_0)}{z - z_0} \Big| = + \infty $$ 
$$ \text{for every} \; z_0 \in \partial \Omega \} $$

is $G_{\delta}$ - dense in $(A(\Omega), || \cdot ||_{\infty})$.

\end{example}

\begin{proof} We consider the following classes of functions 
$$ S_i(\Omega, [a_i, a_{i + 1}]) = \{ f \in A(\Omega): \limsup_{\substack{z \to z_0 \\ z \in [a_i, a_{i + 1}] \setminus \{ z_0 \}}} \Big| \frac{f(z) - f(z_0)}{z - z_0} \Big| = + \infty $$ 
$$ \text{for every} \; z_0 \in [a_i, a_{i + 1}] \} $$

for every $i = 0, 1, \cdots, n - 1$ and
$$ S_n(\Omega, [a_n, a_0]) = \{ f \in A(\Omega): \limsup_{\substack{z \to z_0 \\ z \in [a_n, a_0] \setminus \{ z_0 \}}} \Big| \frac{f(z) - f(z_0)}{z - z_0} \Big| = + \infty $$
$$ \text{for every} \; z_0 \in [a_n, a_0] \}. $$

Let also $H_i$ be the (open) half planes with boundary the lines containing the line segments $[a_i, a_{i + 1}]$ for every $i = 0, 1, \cdots, n - 1$, such that $H_i^{\circ} \cap \Omega \neq \emptyset$. In the same way, we can also define the set $H_n$. Thus
$$ \Omega = \bigcap_{i = 0}^{n} H_i^{\circ}. $$ 

It is easy to prove (in the same way as in Example \ref{third example}) that there exist functions $f_i \in S_i(\Omega, [a_i, a_{i + 1}])$ for every $i = 0, 1, \cdots, n - 1$. Thus, each $S_i(\Omega, [a_i, a_{i + 1}])$ is $G_{\delta}$ - dense in $A(\Omega)$ (if $i = n$, obviously we consider the class $S_n(\Omega, [a_n, a_0])$) and thus, the class 
$$ \Big( \bigcap_{i = 0}^{n - 1} S_i(\Omega, [a_i, a_{i + 1}]) \Big) \cap S_n(\Omega, [a_n, a_0]) \subseteq S(\Omega, \partial \Omega) $$ 

is $G_{\delta}$ - dense in $(A(\Omega)$. Therefore, $S(\Omega, \partial \Omega) \neq \emptyset$. It follows that the class $S(\Omega, \partial \Omega)$ is $G_{\delta}$ - dense in $(A(\Omega)$.

\end{proof}

\begin{example} \label{fifth example}

We consider a family of sets $X_i, i = 1, \cdots, N$ where each $X_i$ is either an open disc, or the interior of a half plane. We set
$$ \Omega = \bigcap_{i = 1}^{N} X_i $$

and we suppose that it holds $\Omega \neq \emptyset$. Let also $\emptyset \neq J \subseteq \partial \Omega$ be a compact set without isolated points. Then, the class of functions 
$$ S = S(\Omega, J) = \{ f \in A(\Omega): \limsup_{\substack{z \to z_0 \\ z \in J \setminus \{ z_0 \}}} \Big| \frac{f(z) - f(z_0)}{z - z_0} \Big| = + \infty \; \text{for every} \; z_0 \in J \} $$

is $G_{\delta}$ - dense in $A(\Omega) = A(\overline{\Omega})$.

\end{example}

\begin{proof} We notice that it holds 
$$ J = \bigcup_{k = 1}^{m} J_m $$

where each $J_m \subseteq \partial \Omega$ is either a line segment or an arc. We consider the following classes of functions
$$ S_k = S_k(\Omega, J_k) = \{ f \in A(\Omega): \limsup_{\substack{z \to z_0 \\ z \in J \setminus \{ z_0 \}}} \Big| \frac{f(z) - f(z_0)}{z - z_0} \Big| = + \infty \; \text{for every} \; z_0 \in J_k \}. $$

Obviously, it holds 
$$ \bigcap_{k = 1}^{m} S_k \subseteq S. $$

In the same way as in Example \ref{third example} we obtain that it holds $S_k(\Omega, J_k) \neq \emptyset$ for every $k = 1, \cdots, m$. Therefore, according to Sections 2 and 3, we obtain that each $S_k(\Omega, J_k)$ is $G_{\delta}$ - dense in $A(\Omega)$. Thus, from Baire's Theorem we obtain that the class 
$$ \bigcap_{k = 1}^{m} S_k $$ 

is also $G_{\delta}$ - dense in $A(\Omega)$ which in turn implies that $S(\Omega, J) \neq \emptyset$. It follows that $S(\Omega, J)$ is also $G_{\delta}$ - dense in $A(\Omega)$.

\end{proof}

\begin{example} \label{sixth example}
Let $\Omega$ be the open set that is the intersection of the interior of a square with side = 2 centered at 0, with the complement of the open unit disc (also centered at $0$). Then, the class of functions 
$$ S(\Omega, \partial \Omega) = \{ f \in A(\Omega): \limsup_{\substack{z \to z_0 \\ z \in \partial \Omega \setminus \{ z_0 \}}} \Big| \frac{f(z) - f(z_0)}{z - z_0} \Big| = + \infty \; \text{for every} \; z_0 \in \partial \Omega \} $$

is $G_{\delta}$ - dense in $A(\Omega)$.

\end{example}

\begin{proof} We notice that it holds $\partial \Omega = \mathbb{T} \cup [a_0, a_1] \cup [a_1, a_2] \cup [a_2, a_3] \cup [a_3, a_0]$, where, of course, the sets $[a_0, a_1], [a_1, a_2], [a_2, a_3]$ and $[a_3, a_0]$ are the four sides of the square. We consider the following classes of functions 
$$ S_i(\Omega, [a_i, a_{i + 1}]) = \{ f \in A(\Omega): \limsup_{\substack{z \to z_0 \\ z \in [a_i, a_{i + 1}] \setminus \{ z_0 \}}} \Big| \frac{f(z) - f(z_0)}{z - z_0} \Big| = + \infty $$
$$ \text{for every} \; z_0 \in [a_i, a_{i + 1}] \} $$

for every $i = 0, 1, 2$ and
$$ S_3(\Omega, [a_3, a_0]) = \{ f \in A(\Omega): \limsup_{\substack{z \to z_0 \\ z \in [a_3, a_0] \setminus \{ z_0 \}}} \Big| \frac{f(z) - f(z_0)}{z - z_0} \Big| = + \infty $$
$$ \text{for every} \; z_0 \in [a_3, a_0] \}. $$ 

Then, since the set $(\mathbb{C} \cup \{ \infty \}) \setminus \overline{\Omega}$ has a finite number of connected components, according to \cite{RUDIN} it holds $A(\Omega) = R(\overline{\Omega})$.

Using the results of Section 2 and in the same way as in Example \ref{third example} we obtain that the classes $S_0(\Omega, [a_0, a_1]), \cdots, S_3([a_3, a_0])$ are $G_{\delta}$ - dense in $A(\Omega)$.

Also, we consider the following class of functions 
$$ B(\Omega, \mathbb{T}) = \{ f \in A(\Omega): \limsup_{\substack{z \to z_0 \\ z \in \mathbb{T} \setminus \{ z_0 \}}} \Big| \frac{f(z) - f(z_0)}{z - z_0} \Big| = + \infty \; \text{for every} \; z_0 \in \mathbb{T} \} $$

and let $g$ be the function defined in Example \ref{first example}. We consider the following function $h: (\mathbb{C} \setminus \overline{D(0, 1)}) \to \mathbb{C}$ with $h(w) = g(\frac{1}{w})$ for every $w$. Then it holds $h \in A(\mathbb{C} \setminus \overline{D(0, 1)})$ and is also easy to verify that $(h \restriction_{\Omega}) \in B(\Omega, \mathbb{T})$. Thus, the class $B(\Omega, \mathbb{T})$ is $G_{\delta}$ - dense in $A(\Omega)$ and since $S_0([a_0, a_1]) \cap S_1([a_1, a_2]) \cap S_2([a_2, a_3]) \cap S_3([a_3, a_0]) \cap B(\Omega, \mathbb{T}) \subseteq S(\Omega, \partial \Omega)$ it follows that $S(\Omega, \partial \Omega) \neq \emptyset$ and thus, the class $S(\Omega, \partial \Omega)$ is $G_{\delta}$ - dense in $(A(\Omega)$.

\end{proof}

\begin{remark} Notice that in the previous example the set $\Omega$ is not a convex one.

\end{remark}

\begin{remark} In the previous example, by considering the function $h(w) = g(\frac{1}{w})$ we obtain that $S(\Omega, \mathbb{T}) \neq \emptyset$, where $\Omega = \{ z \in \mathbb{C}: |z| > 1 \}$. In addition, it is obvious that if $S(\Omega, J) \neq \emptyset$, where $G \subseteq \Omega$ is an open set and $J \subseteq \partial \Omega \cap \partial G$, then it holds $S(G, J) \neq \emptyset$.

\end{remark}

\begin{example} \label{seventh example}
 
Let $\emptyset \neq \Omega \subseteq \mathbb{C}$ be an open set and $\emptyset \neq J \subseteq \partial \Omega$ be a compact set without isolated points. Let also $V \subseteq \mathbb{C}$ be an open set with $J \subseteq V$. We suppose that there exists a function $\phi: \Omega \cup V \to \mathbb{C}$, which is 1 -1, holomorphic and can be extended in a continuous way to a homeomorphism from $\overline{\Omega} \cup \overline{V}$ to $\mathbb{C}$. We consider the following classes of functions
$$ S_1(\Omega, J) = \{ f \in A(\Omega): \limsup_{\substack{z \to z_0 \\ z \in J \setminus \{ z_0 \}}} \Big| \frac{f(z) - f(z_0)}{z - z_0} \Big| = + \infty \; \text{for every} \; z_0 \in J \} $$

and
$$ S_2(G, \widetilde{J}) = \{ f \in A(G): \limsup_{\substack{z \to z_0 \\ z \in \widetilde{J} \setminus \{ z_0 \}}} \Big| \frac{f(z) - f(z_0)}{z - z_0} \Big| = + \infty \; \text{for every} \; z_0 \in \tilde{J} \} $$

where $G = \phi(\Omega)$ and $\widetilde{J} = \phi(J)$. Then, if $S_1(\Omega, J) \neq \emptyset$ then it holds $S_2(\Omega, \widetilde{J}) \neq \emptyset$.

\end{example}

\begin{proof} It is well known that the set $G$ is an open one, $\phi'(z) \neq 0$ for every $z \in \Omega \cup V$, the function $\phi^{- 1}: \phi(\Omega \cup V) \to \Omega \cup V$ is a holomorphic function with $(\phi^{- 1})'(w) \neq 0$ for every $w \in \phi(\Omega \cup V)$ and $\phi^{- 1}$ can be extended continuously to a 1 - 1 function from $\overline{\phi(\Omega \cup V)}$ to $\mathbb{C}$. Now let $f \in S_1(\Omega, J)$. We consider the function $h = f \circ \phi^{- 1}: \phi(\Omega) \to \mathbb{C}$. Obviously, it holds $h \in A(G)$. In addition, $\widetilde{J}$ is a compact set without isolated points since $\phi$ is a homeomorphism. Let $\omega_0 = \phi(z_0) \in \widetilde{J}$, for a $z_0 \in J$. Then, there exists a sequence $\{ z_n \}_{n \geq 1} \in J \setminus \{ z_0 \}$ with $z_n \to z_0$ satisfying
$$ \lim_{n \to + \infty} \Big| \frac{f(z_n) - f(z_0)}{z_n - z_0} \Big| = + \infty. $$

We set $\omega_n = \phi(z_n) \in \widetilde{J} \setminus \{ w_0 \}$ for every $n \geq 1$. Then, it holds $\omega_n \to \omega_0$. In addition, we obtain
$$ \Big| \frac{h(\omega_n) - h(\omega_0)}{\omega_n - \omega_0} \Big| = \Big| \frac{f(z_n) - f(z_0)}{\phi(z_n) - \phi(z_0)} \Big| = \Big| \frac{f(z_n) - f(z_0)}{z_n - z_0} \Big| \cdot \Big| \frac{z_n - z_0}{\phi(z_n) - \phi(z_0)} \Big|. $$

We have
$$ \lim_{n \to + \infty} \Big| \frac{z_n - z_0}{\phi(z_n) - \phi(z_0)} \Big| = \frac{1}{|\phi'(z_0)|} $$

with $\phi'(z_0) \neq 0$, since $z_0 \in \Omega \cup V$. Thus, 
$$ \lim_{n \to + \infty} \Big| \frac{h(\omega_n) - h(\omega_0)}{\omega_n - \omega_0} \Big| = + \infty $$ 

which it turn implies that 
$$ \limsup_{\substack{\omega \to \omega_0 \\ \omega \in \widetilde{J} \setminus \{ \omega_0 \}}} \Big| \frac{h(\omega) - h(\omega_0)}{\omega - \omega_0} \Big| = + \infty $$ 

for every $\omega_0 \in \widetilde{J}$. Thus, $h \in S_2(G, \widetilde{J})$ and that completes the proof.

\end{proof}

\begin{remark} If the class $S_1(\Omega, J)$ is $G_{\delta}$ - dense, then the class $S_2(G, \widetilde{J})$ is also $G_{\delta}$ - dense, since the function $\psi: A(\Omega) \to A(\phi(\Omega))$ with $\psi(f) = f \circ \phi^{- 1}$ for every $f \in A(\Omega)$ is a homeomorphism, where the spaces $A(\Omega)$ and $A(\phi(\Omega))$ are endowed with their natural topologies.                                   Notice that $\psi(S_1(\Omega, J)) = S_2(G, \widetilde{J})$.

Especially, if $D$ is the open unit disc, $J = \mathbb{T}$, $V = D(0, r)$ with $r > 1$ and we consider the function $\phi: D \cup V \equiv V \to \mathbb{C}$ which is 1 - 1 and holomorphic, then the respective class $S_2(G, \widetilde{J})$ in $A(\phi(D))$ with $\widetilde{J} = \phi(\partial D)$ is $G_{\delta}$ - dense in $A(\phi(D))$ with $\widetilde{J}$ being not - necessarily a line segment, an arc or the union of the previous two.

\end{remark}

\begin{example} \label{eighth example}

We consider the open set $\Omega = D \setminus [0, \frac{1}{2}]$. Then the class of functions 
$$ S(\Omega, \partial \Omega) = \{ f \in A(\Omega): \limsup_{\substack{\omega \to \omega_0 \\ \omega \in \partial \Omega \setminus \{ \omega_0 \}}} \Big| \frac{f(z) - f(z_0)}{z - z_0} \Big| = + \infty \; \text{for every} \; z_0 \in \partial \Omega \} $$

is void.

\end{example}

\begin{proof} Let $f \in A(\Omega)$. Then the function $f$ is continuous on $\overline{D}$ and holomorphic in $D \setminus \mathbb{R}$. From a known corollary of Morera's Theorem, it follows that $f$ is also holomorphic in $D$. Thus, 
$$ \limsup_{\substack{z \to z_0 \\ z \in [0, \frac{1}{2}] \setminus \{ z_0 \}}} \Big| \frac{f(z) - f(z_0)}{z - z_0} \Big| = |f'(z_0)| < + \infty $$ 

for every $z_0 \in [0, \frac{1}{2}]$, therefore $f \not\in S(\Omega, \partial \Omega)$. It follows that $S(\Omega, \partial \Omega) = \emptyset$.

\end{proof}

\begin{remark} From the previous example we deduce that the class $S(\Omega, \partial \Omega)$ is not always $G_{\delta}$ - dense, even if the set $\Omega$ is a bounded set.

\end{remark}

\begin{example} \label{ninth example}

We consider the open set $\Omega = \{ z \in \mathbb{C}: Re(z) > 0 \} \setminus [1, + \infty]$ and the respective class $S(\Omega, J)$, where $J = [1, + \infty)$. Then it holds $S(\Omega, J) = \emptyset$.

\end{example}

\begin{proof} Let $f \in A(\Omega)$. In the same way as in the previous example, the function $f$ can be extended holomorphically on $H = \{ z: Re(z) > 0 \}$. Thus 
$$ \limsup_{\substack{z \to z_0 \\ z \in J \setminus \{ z_0 \}}} \Big| \frac{f(z) - f(z_0)}{z - z_0} \Big| = |f'(z_0)| < + \infty $$

for every $z_0 \in J$. Therefore, it holds $f \not\in S(\Omega, J)$ and thus $S(\Omega, J) = \emptyset$.

\end{proof}

\begin{remark} The previous example shows that the class $S(\Omega, J)$ can be also void even if the set $\Omega$ is an unbounded one and the set $J$ is not necessarily a compact one.

\end{remark}

\begin{remark} So far, we do not have an example where it holds $S(\Omega, J) = \emptyset$ and $J \subseteq \partial (\overline{\Omega})$.

\end{remark}

\begin{example} \label{tenth example}

So far, we have seen that if $\phi: D(0, r) \to \mathbb{C}$ is a holomorphic and 1 - 1 function and  $r > 1$, then for $\Omega = \phi(D(0, 1))$ and $J$ any compact subset of $\partial \Omega$ it holds $S(\Omega, J) \neq \emptyset$ (and therefore, the class $S(\Omega, J)$ is $G_{\delta}$ - dense in $A(\Omega) = R(\Omega)$). Under the previous assumptions it holds $\phi'(e^{i \theta}) \neq 0$ for every $\theta \in \mathbb{R}$. We note that the same result holds if $\Omega$ is a Jordan domain and $\phi: D(0, 1) \to \Omega$ is holomorphic, 1 - 1 and onto, $\phi'$ extends continuously on $\overline{D(0, 1)}$ and $\phi'(e^{i \theta}) \neq 0$ for every $\theta \in \mathbb{R}$. We remind that $\phi$ extends to a homeomorphism from $\overline{D(0, 1)}$ to $\overline{\Omega}$, according to Osgood - Caratheodory Theorem \cite{KOOSIS}. The proof that $S(\Omega, J) \neq \emptyset$ is similar to the proof of Example \ref{seventh example} and therefore is omitted. 

\end{example}

\noindent
\textbf{Acknowledgment.} The authors would like to thank Professor Vassili Nestoridis for suggesting the specific topic and also for his valuable guidance during the creation of this paper. The second author acknowledges financial support from Program 70/3/13297 from ELKE, University of Athens, Greece.

\bigskip

\noindent
K. Kavvadias

\medskip

\noindent
National and Kapodistrian University of Athens \\
Department of Mathematics\\
Panepistemiopolis \\
157 84 Athens\\
Greece\\
e-mail: kavvadiaskostantinos@hotmail.com

\bigskip
\bigskip
\bigskip

\noindent
K. Makridis

\medskip

\noindent
National and Kapodistrian University of Athens \\
Department of Mathematics\\
Panepistemiopolis \\
157 84 Athens\\
Greece\\
e-mail: kmak167@gmail.com

\end{document}